\documentclass[10pt]{article}

\usepackage[margin=1in]{geometry}

\usepackage{amsmath}
\usepackage{amssymb}
\usepackage{amsthm}
\usepackage{url}
\usepackage{mathtools}
\usepackage{MnSymbol}
\usepackage{todonotes}

\usepackage{hyperref}
\usepackage{cleveref}
\usepackage[shortlabels]{enumitem}

\usepackage[all,cmtip]{xy}

\newtheorem{theorem}{Theorem}[section]
\newtheorem{lemma}[theorem]{Lemma}

\newtheorem{corollary}[theorem]{Corollary}

\newcounter{claims}[theorem]
\newtheorem{claim}[claims]{Claim}

\theoremstyle{definition}
\newtheorem{definition}[theorem]{Definition}

\newtheorem{question}[theorem]{Question}

\theoremstyle{remark}

\newcommand{\la}{\langle}
\newcommand{\ra}{\rangle}

\DeclareMathOperator{\dom}{dom}

\newcommand{\concat}[0]{\textrm{\^{}}}

\newcommand{\Q}{\mathbb{Q}}

\newcommand{\uhr}[1]{\! \upharpoonright_{#1}}

\renewcommand \leq {\leqslant}
\renewcommand \geq {\geqslant}
\renewcommand \le {\leqslant}
\renewcommand \ge {\geqslant}
\newcommand{\converge}{\!\!\downarrow}

\renewcommand \phi {\varphi}
\newcommand{\cat}{\widehat{\phantom{\alpha}}}

\newcommand{\Psf}{\mathsf{P}}
\newcommand{\Qsf}{\mathsf{Q}}

\newcommand{\Ccal}{\mathcal{C}}
\newcommand{\Dcal}{\mathcal{D}}
\newcommand{\Ecal}{\mathcal{E}}

\newcommand{\Lcal}{\mathcal{L}}
\newcommand{\Mcal}{\mathcal{M}}

\newcommand{\Scal}{\mathcal{S}}

\newcommand{\Wcal}{\mathcal{W}}

\newcommand{\uh}{{\upharpoonright}}

\renewcommand{\setminus}{\smallsetminus}

\def\qt#1{``#1''}%

\newcommand{\sss}[1]{\ensuremath{\sf{#1}}}

\DeclareMathOperator{\rca}{\sss{RCA}_0}

\DeclareMathOperator{\aca}{\sss{ACA}}
\DeclareMathOperator{\wkl}{\sss{WKL}}

\DeclareMathOperator{\rt}{\sss{RT}}

\DeclareMathOperator{\ads}{\sss{ADS}}

\DeclareMathOperator{\coh}{\sss{COH}}

\DeclareMathOperator{\ts}{\sss{TS}}

\DeclareMathOperator{\emo}{\sss{EM}}

\begin{document}

\title{Relationships between computability-theoretic properties of problems}
\author{ \and Rod Downey \and Noam Greenberg \and Matthew Harrison-Trainor \and Ludovic Patey \and Dan Turetsky}

\makeatletter
\def\blfootnote{\xdef\@thefnmark{}\@footnotetext}
\makeatother

\maketitle

\begin{abstract}
A problem is a multivalued function from a set of \emph{instances} to a set of \emph{solutions}. We consider only instances and solutions coded by sets of integers. A problem admits preservation of some computability-theoretic weakness property if every computable instance of the problem admits a solution relative to which the property holds. For example, cone avoidance is the ability, given a non-computable set $A$ and a computable instance of a problem $\Psf$, to find a solution relative to which $A$ is still non-computable. 

In this article, we compare relativized versions of computability-theoretic notions of preservation which have been studied in reverse mathematics, and prove that the ones which were not already separated by natural statements in the literature actually coincide.  In particular, we prove that it is equivalent to admit avoidance of 1 cone, of $\omega$ cones, of 1 hyperimmunity or of 1 non-$\Sigma^0_1$ definition.
We also prove that the hierarchies of preservation of hyperimmunity and non-$\Sigma^0_1$ definitions coincide. On the other hand, none of these notions coincide in a non-relativized setting.
\end{abstract}

\section{Introduction}

In this article, we classify computability-theoretic preservation properties studied in reverse mathematics, namely cone avoidance, preservation of hyperimmunities, preservation of non-$\Sigma^0_1$ definitions, among others. Many of these preservation properties have already been separated using natural problems in reverse mathematics -- that is, there is a natural problem which is known to admit preservation of one property but not preservation of the other.  The observation that emerges from our work is that those properties which have not already been separated in fact coincide.\footnote{The authors are thankful to Mariya Soskova for interesting comments and discussions about cototal degrees.}

Reverse mathematics is a foundational program whose goal is to determine the optimal axioms for proving ordinary theorems. It uses subsystems of second-order arithmetics, with a base theory, $\rca$, capturing \emph{computable mathematics}. See Simpson's book~\cite{Simpson2009Subsystems} for a reference in reverse mathematics. A structure in this language is a tuple $(N, \Scal, +_N, *_N, <_N, 0_N, 1_N)$, where $N$ stands for the first-order part, and $\Scal$ for the set of reals. We are in particular interested in structures in which the first-order part consists of the standard integers $\omega$, equipped with the natural operations. These structures are called \emph{$\omega$-structures}, and are fully specified by their second-order part $\Scal$. The choice of the axioms of $\rca$ yields a nice characterization of the second-order part of $\omega$-models of $\rca$ in terms of \emph{Turing ideals}.

\begin{definition}
A \emph{Turing ideal} is a collection of reals $\Scal \subseteq 2^\omega$ which is closed under the effective join and downward-closed under the Turing reduction. In other words
\begin{itemize}
	\item[(a)] $\forall X, Y \in \Scal, X \oplus Y = \{2n : n \in X\} \cup \{2n+1 : n \in Y\} \in \Scal$
	\item[(b)] $\forall X \in \Scal, \forall Y \leq_T X, Y \in \Scal$ 
\end{itemize}
\end{definition}

Many statements studied in reverse mathematics can be formulated as \emph{mathematical problems}, with instances and solutions. For example, weak K\"onig's lemma ($\wkl$) asserts that every infinite, finitely branching subtree of $2^{<\omega}$ has an infinite path. Here, an instance is such a tree $T$, and a solution to $T$ is an infinite path through it. An $\omega$-structure $\Mcal$ with second-order part $\Scal$ is a \emph{model} of a problem $\Psf$ (written $\Mcal \models \Psf$) if every instance in $\Scal$ has a solution in it.  In this case we also say that $\Psf$ holds in $\Scal$. In order to separate a problem $\Psf$ from another problem $\Qsf$ in reverse mathematics, one usually constructs a Turing ideal $\Scal$ in which $\Psf$ holds, but not $\Qsf$. However, when closing the Turing ideal with solution to instances of $\Psf$, one must be careful not to make it a model of $\Qsf$. This motivates the use of preservation properties.

\begin{definition}\label{def:preservation_of_W}
Fix a collection of sets $\Wcal \subseteq 2^\omega$ downward-closed under Turing reduction. A problem $\Psf$ \emph{admits preservation of $\Wcal$} if for every set $Z \in \Wcal$ and every $Z$-computable instance $X$ of $\Psf$, there is a solution $Y$ to $X$ such that $Z \oplus Y \in \Wcal$.
\end{definition}

The following basic lemma is at the core of separations in reverse mathematics.

\begin{lemma}
Suppose a problem $\Psf$ admits preservation of some collection $\Wcal$, but another problem $\Qsf$ does not. Then there is a Turing ideal $\Scal \subseteq \Wcal$ in which $\Psf$ holds, but not $\Qsf$.
\end{lemma}
\begin{proof}
Since $\Qsf$ does not admit preservation of $\Wcal$, there is some $Z \in \Wcal$, and a $\Qsf$-instance $X_\Qsf \leq_T Z$ such that for every solution $Y$ to $X_\Qsf$, $Z \oplus Y \not \in \Wcal$. We will build a Turing ideal $\Scal \subseteq \Wcal$ containing $Z$ and in which $\Psf$ holds. In particular, $\Qsf$ cannot hold in any such Turing ideal.
We build a countable sequence of sets $Z_0, Z_1, \dots$ such that for every $n \in \omega$, $\bigoplus_{s < n} Z_s \in \Wcal$, and for every $\Psf$-instance $X \leq_T \bigoplus_{s < n} Z_s$, there is some $m \in \omega$ such that $Z_m$ is a $\Psf$-solution to $X$. Start with $Z_0 = Z$. 
Having defined $Z_0, \dots, Z_{n-1}$, pick the next $\Psf$-instance $X \leq_T \bigoplus_{s < n} Z_s$ by ensuring that each instance will receive attention at a finite stage. Since $\Psf$ admits preservation of $\Wcal$, there is a $\Psf$-solution $Z_n$ to $X$ such that $\bigoplus_{s \leq n} Z_s \in \Wcal_n$. Then go to the next stage.
The collection $\Scal = \{ X : (\exists s) X \leq_T \bigoplus_{s < n} Z_s \}$ is a Turing ideal included in $\Wcal$ in which $\Psf$ holds but not~$\Qsf$.
\end{proof}

Many statements in reverse mathematics, mostly coming from Ramsey theory, have been separated by looking at the appropriate computability-theoretic notion of preservation. We now detail some outstanding ones, which will serve as a basis for our classification study.

\subsection{Cone avoidance}

Perhaps the most important property of preservation in reverse mathematics is the notion of \emph{cone avoidance}, both for its intrinsic significance, namely, the inability to code sets into the solutions of a computable instance of a problem, and as a tool to separate statements from the Arithmetic Comprehension Axiom ($\aca$).

\begin{definition}
Fix $n \leq \omega$. A problem $\Psf$ \emph{admits avoidance of $n$ cones} if for every set $Z$ and every collection $\{ B_s : s < n \}$ of non-$Z$-computable sets,
	every $\Psf$-instance $X \leq_T Z$ has a solution $Y$ such that for every $s < n$, $B_s \not \leq_T Z \oplus Y$.
\end{definition}
This definition can be understood in terms of \Cref{def:preservation_of_W} by defining
\[
\Wcal(B_s : s < n) = \{ Y : (\forall s < n)\, B_s \not \le_T Y\}.
\]
Then $\Psf$ admits avoidance of $n$ cones precisely if it admits preservation of $\Wcal(B_s : s < n)$ for every collection $\{B_s : s < n\}$. 
A similar analysis applies to all of the avoidance properties we will study, although for the rest we will not take the time to make it explicit.

Jockusch and Soare~\cite[Theorem 2.5]{Jockusch197201} proved that weak K\"onig's lemma ($\wkl$) admits avoidance of $\omega$ cones.\footnote{Avoidance of $\omega$ cones is known as \emph{cone avoidance} in the literature.} Seetapun's celebrated theorem (see~\cite{Seetapun1995strength}) states that Ramsey's theorem for pairs  ($\rt^2_2$) admits avoidance of $\omega$ cones, answering a long-standing open question. On the other hand, Jockusch~\cite{Jockusch1972Ramseys} proved that Ramsey's theorem for triples ($\rt^3_2$) does not. Later, Wang~\cite{Wang2014Some} proved the surprising result that for every $n \geq 2$, there is some $k \in \omega$ such that $\rt^n_{k+1,k}$ admits avoidance of $\omega$ cones, where $\rt^n_{k+1,k}$ asserts that for every coloring $f : [\omega]^n \to k+1$, there is an infinite set $H \subseteq \omega$ such that $|f[H]^n| \leq k$. By looking at the literature, one can observe that all the proofs of cone avoidance hold for $\omega$ cones simultaneously. In this paper, we justify this observation by proving that avoidance of 1 cone and of $\omega$ cones coincide.

\subsection{Preservation of non-$\Sigma^0_1$ definitions}

Wang~\cite{Wang2014Definability} dramatically simplified separation proofs of the Erd\H{o}s-Moser ($\emo$) from the Ascending Descending Sequence ($\ads$) of Lerman, Solomon and Towsner~\cite{Lerman2013Separating} by proving that some problems \qt{preserve} the arithmetical hierarchy, in the sense that given a fixed strictly non-$\Sigma^0_n$ set $A$ and given a $\Psf$-instance, there is a solution $Y$ such that $A$ is not $\Sigma^0_n(Y)$. We consider the case of non-$\Sigma^0_1$ sets.

\begin{definition}
Fix $n \leq \omega$. A problem $\Psf$ \emph{admits preservation of $n$ non-$\Sigma^0_1$ definitions} if for every set $Z$ and every collection $\{ B_s : s < n \}$ of non-$Z$-c.e.\ sets,
	every $\Psf$-instance $X \leq_T Z$ has a solution $Y$ such that for every $s < n$, $B_s$ is not $Z \oplus Y$-c.e.
\end{definition}

This framework was very successful in proving separation results between Ramsey-like statements over $\omega$-models.
Wang~\cite{Wang2014Definability} proved that $\wkl$ and the Erd\H{o}s-Moser theorem ($\emo$) admits preservation of $\omega$ non-$\Sigma^0_1$ definitions, while the thin set theorem for pairs ($\ts^2_\omega$) does not. 
Patey~\cite{Patey2016weakness} proved that for every $k \geq 1$, $\rt^2_{k+1,k}$ admits preservation of $k$ but not $k+1$ non-$\Sigma^0_1$ definitions. In particular, Ramsey's theorem for pairs and two colors admits preservation of 1 but not 2 non-$\Sigma^0_1$ definitions.

\subsection{Preservation of hyperimmunities}

The proof that Ramsey's theorem for triples does not admit cone avoidance consists of constructing a computable coloring $f : [\omega]^3 \to 2$ such that every $f$-homogeneous set $H = \{x_0 < x_1 < \dots \}$ is so sparse that its \emph{principal function} $p_H : \omega \to \omega$ defined by $p_H(n) = x_n$ grows faster than the settling time of the halting set. Actually, all the proofs that a Ramsey-like statement does not admit cone avoidance exploit the existence of instances whose solutions are all sufficiently sparse to compute fast-growing functions dominating moduli of computation~\cite{Patey2019Ramsey}. It is therefore natural to consider which problems have the ability to compute fast-growing functions.

A function $f : \omega \to \omega$ is \emph{hyperimmune} if it is not dominated by any computable function. An infinite set $A = \{x_0 < x_1 < \dots \}$ is \emph{hyperimmune} if its principal function $p_A$ is hyperimmune. Equivalently, a set $A$ is hyperimmune if for every computable sequence of pairwise disjoint non-empty finite coded sets $F_0, F_1, \dots$, there is some $n \in \omega$ such that $A \cap F_n = \emptyset$.

\begin{definition}
Fix $n \leq \omega$. A problem $\Psf$ \emph{admits preservation of $n$ hyperimmunities} if for every set $Z$ and every collection $\{ f_s : s < n \}$ of $Z$-hyperimmune functions,
	every $\Psf$-instance $X \leq_T Z$ has a solution $Y$ such that for every $s < n$, $f_s$ is $Z \oplus Y$-hyperimmune.
\end{definition}

Jockusch and Soare~\cite[Theorem 2.4]{Jockusch197201} proved that $\wkl$ admits preservation of $\omega$ hyperimmunities (in fact, every computable instance of $\wkl$ has a solution of hyperimmune-free degree). Patey~\cite{Patey2017Iterative} proved that the Erd\H{o}s-Moser theorem admits preservation of $\omega$ hyperimmunities and that for every $k \geq 1$, $\rt^2_{k+1,k}$ admits preservation of $k$, but not $k+1$, hyperimmunities. He also proved that the thin set theorem for pairs admits preservation of $k$ hyperimmunities for every $k \in \omega$, but not of $\omega$ hyperimmunities.

As it happens, all the separations over $\omega$-models and over computable reduction which have been proven by notions of preservation of non-$\Sigma^0_1$ definitions can also be proved by preservation of hyperimmunities, and vice versa. We prove in this paper that this is not a coincidence, and that the two notions of preservation are indeed equivalent.

\subsection{Constant-bound trace avoidance}

Both the original proof of cone avoidance of Ramsey's theorem for pairs by Seetapun~\cite{Seetapun1995strength} and the proof by Cholak, Jockusch and Slaman~\cite{Cholak2001strength} involve Mathias-like notions of forcing within models of weak K\"onig's lemma. Their proofs seem to make an essential use of compactness, and the community naturally wondered whether this use was necessary. Liu~\cite{Liu2012RT22} recently negatively answered the long-standing open question of whether Ramsey's theorem for pairs implies weak K\"onig's lemma in reverse mathematics. He later~\cite{Liu2015Cone} refined his argument and proved that $\rt^2_2$ does not even imply the existence of Martin-L\"of randoms, using the notion of constant-bound trace avoidance for closed sets\footnote{In his article, Liu calls this notion \emph{constant-bound enumeration avoidance}. We rechristen it in keeping with the notion of traces as studied in algorithmic randomness~\cite{TeZa01}}. 

Given a closed set $\Ccal \subseteq 2^\omega$, a \emph{trace} is a collection
of finite coded sets of strings $F_0, F_1, \dots$ such that for every $n \in \omega$, 
$F_n$ contains only strings of length exactly $n$, and $\Ccal \cap [F_n] \neq \emptyset$ where $[F_n]$ is the clopen set generated by $F_n$. In other words, for every $n \in \omega$, there is a string $\sigma \in F_n$ with $|\sigma| = n$ such that $\sigma \prec P$ for some $P \in \Ccal$. A \emph{$k$-trace} of $\Ccal$ is a trace such that $|F_n| = k$ for every $n \in \omega$. A \emph{constant-bound trace} of $\Ccal$ is a $k$-trace for some $k \in \omega$.

\begin{definition}
Fix $n \leq \omega$. A problem $\Psf$ \emph{admits avoidance of constant-bound traces for $n$ closed sets} if for every set $Z$ and every collection of closed sets $\{ \Ccal_s \subseteq 2^\omega : s < n \}$ with no $Z$-computable constant-bound trace, 
	every $\Psf$-instance $X \leq_T Z$ has a solution $Y$ such that for every $s < n$, $\Ccal_s$ has no $Z \oplus Y$-computable constant-bound trace.
\end{definition}

This notion of avoidance, which at first sight seems slightly more artificial, happens to be a very powerful tool to prove that Ramsey-like statements do not imply notions of compactness.

Liu~\cite{Liu2015Cone} proved that Ramsey's theorem for pairs and two colors ($\rt^2_2$) admits avoidance of constant-bound traces for 1 closed set, while weak K\"onig's lemma ($\wkl$) does not. Patey~\cite{PateyCombinatorial} proved that the Erd\H{o}s-Moser theorem ($\emo$) admits avoidance of constant-bound traces for $\omega$ closed sets, and that for every $k \geq 1$, $\rt^2_{k+1,k}$ admits avoidance of constant-bound traces for $k$ but not $k+1$ closed sets, and that $\ts^2_\omega$ admits avoidance of constant-bound traces for $k$ closed sets for every $k \in \omega$, but not for $\omega$ closed sets.

\subsection{Other preservation notions}

As explained, the notion of hyperimmunity can be expressed both in terms of fast-growing functions, and as sets which cannot be traced by computable strong arrays. Hyperimmunity strengthens another property of sets called \emph{immunity}, which refers to the impossibility of computing an infinite subset of the set. Immunity is a natural notion to look at when considering Ramsey-like theorems, since their sets of solutions are closed under infinite subsets. Although hyperimmunity is a strengthening of immunity, preservation of hyperimmunity is actually strictly weaker than preservation of immunity.

\begin{definition}
Fix $n \leq \omega$. A problem $\Psf$ \emph{admits preservation of $n$ immunities} if for every set $Z$ and every collection $\{ B_s : s < n \}$ of $Z$-immune sets,
	every $\Psf$-instance $X \leq_T Z$ has a solution $Y$ such that for every $s < n$, $B_s$ is $Z \oplus Y$-immune.
\end{definition}

Very few statements in reverse mathematics admit preservation of $\omega$ immunities. The most notable is the cohesiveness principle ($\coh$). All the statements which are known to admit preservation of $\omega$ immunities actually also preserve the following seemingly stronger notion.

\begin{definition}
Fix $n \leq \omega$. A problem $\Psf$ \emph{admits avoidance of $n$ closed sets in the Baire space} if for every set $Z$ and every collection $\{ \Ccal_s : s < n \}$ of closed sets in the Baire space with no $Z$-computable member,
	every $\Psf$-instance $X \leq_T Z$ has a solution $Y$ such that for every $s < n$, $\Ccal_s$ has no $Z \oplus Y$-computable member.
\end{definition}

A similar notion can be defined for closed sets in the Cantor space. We will prove that avoiding closed sets in the Cantor space or in the Baire space are equivalent. We leave open the question whether every problem admitting preservation of $\omega$ immunities also admits avoidance of $\omega$ closed sets.

\subsection{Summary of the relations between properties of preservation}

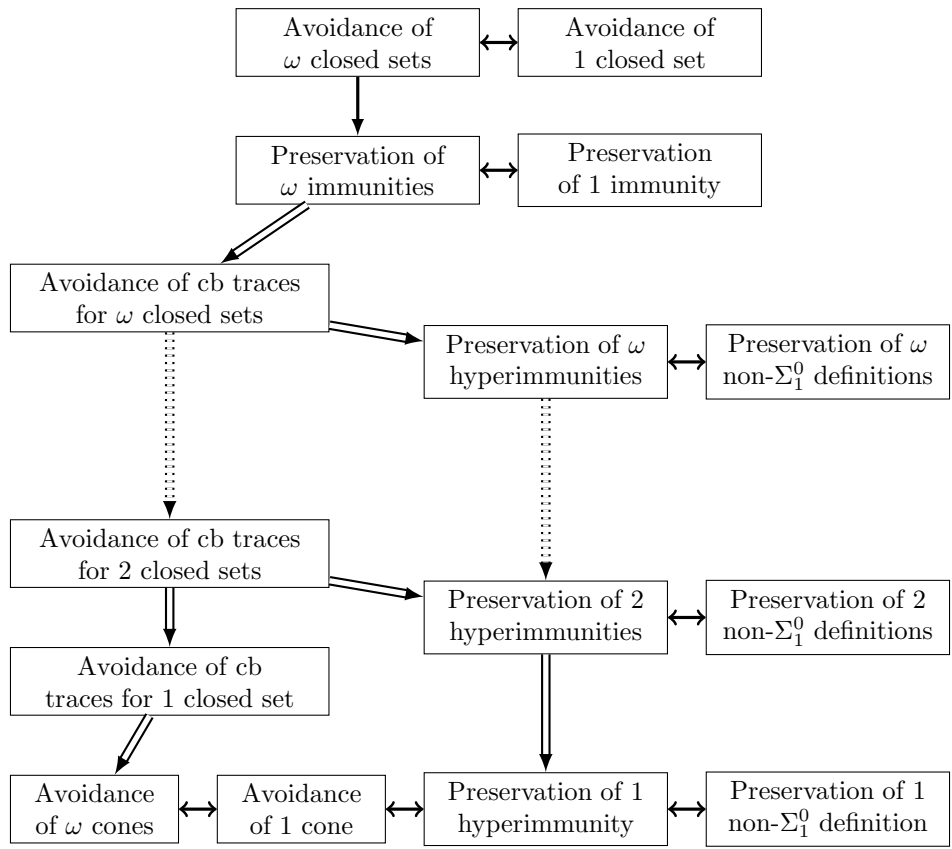
\begin{figure}[htbp]
\begin{center}
\begin{tikzpicture}[x=2.5cm, y=1.7cm, 
	node/.style={draw,minimum size=2em},
	impl/.style={draw,very thick,-latex},
	strict/.style={draw, thick, -latex, double distance=2pt},
	nonimpl/.style={draw, very thick, dotted, -latex}]

	\node[node,text width=3cm,align=center] (omegaclosedset) at (2, 6) {Avoidance of $\omega$ closed sets};
	\node[node,text width=3cm,align=center] (oneclosedset) at (3.5, 6) {Avoidance of 1 closed set};
	\node[node,text width=3cm,align=center] (omegaimmunities) at (2, 5) {Preservation of $\omega$ immunities};
	\node[node,text width=3cm,align=center] (oneimmunity) at (3.5, 5) {Preservation of 1 immunity};
	\node[node,text width=4cm,align=center] (omegacbavoidance) at (1, 4) {Avoidance of cb traces for $\omega$ closed sets};
	\node[node,text width=4cm,align=center] (twocbavoidance) at (1, 2) {Avoidance of cb traces for 2 closed sets};
	\node[node,text width=4cm,align=center] (onecbavoidance) at (1, 1) {Avoidance of cb traces for 1 closed set};
	\node[node,text width=3cm,align=center] (omegahyperimmunities) at (3, 3.5) {Preservation of $\omega$ hyperimmunities};
	\node[node,text width=3cm,align=center] (omeganonce) at (4.5, 3.5) {Preservation of $\omega$ non-$\Sigma^0_1$ definitions};
	\node[node,text width=3cm,align=center] (twohyperimmunities) at (3, 1.5) {Preservation of 2 hyperimmunities};
	\node[node,text width=3cm,align=center] (twononce) at (4.5, 1.5) {Preservation of 2 non-$\Sigma^0_1$ definitions};
	\node[node,text width=3cm,align=center] (onehyperimmunity) at (3, 0) {Preservation of 1 hyperimmunity};
	\node[node,text width=3cm,align=center] (onenonce) at (4.5, 0) {Preservation of 1 non-$\Sigma^0_1$ definition};
	\node[node,text width=2cm,align=center] (onecone) at (1.7, 0) {Avoidance of 1 cone};
\node[node,text width=2cm,align=center] (omegacones) at (0.6, 0) {Avoidance of $\omega$ cones};

	\draw[impl,<->] (omegaclosedset) -- (oneclosedset);
	\draw[impl,<->] (omegaimmunities) -- (oneimmunity);
	\draw[impl] (omegaclosedset) -- (omegaimmunities);
	\draw[strict] (omegaimmunities) -- (omegacbavoidance);
	\draw[strict,dotted] (omegacbavoidance) -- (twocbavoidance);
	\draw[strict] (twocbavoidance) -- (onecbavoidance);
	\draw[impl,<->] (omegahyperimmunities) -- (omeganonce);
	\draw[strict] (omegacbavoidance) -- (omegahyperimmunities);
	\draw[strict,dotted] (omegahyperimmunities) -- (twohyperimmunities);
	\draw[impl,<->] (twohyperimmunities) -- (twononce);
	\draw[strict] (twocbavoidance) -- (twohyperimmunities);
	\draw[impl,<->] (omegacones) -- (onecone);
	\draw[impl,<->] (onecone) -- (onehyperimmunity);
	\draw[impl,<->] (onehyperimmunity) -- (onenonce);
	\draw[strict] (onecbavoidance) -- (omegacones);
	\draw[strict] (twohyperimmunities) -- (onehyperimmunity);

\end{tikzpicture}

\end{center}
\caption{Diagram of relations between properties of preservation. A double arrow denotes a strict implication, a dotted double arrows express a strict hierarchy, while a bidirectional arrow is an equivalence. The only unknown arrow is the reversal from preservation of $\omega$ immunities to avoidance of $\omega$ closed sets.}
\end{figure}

The notions of preservation admit a combinatorial counterpart, in which no effectiveness restriction is imposed on the instance of the problem. This is the notion of \emph{strong preservation}.

\begin{definition}
Fix a collection of sets $\Wcal \subseteq 2^\omega$ downward-closed under the Turing reduction. A problem $\Psf$ \emph{admits strong preservation of $\Wcal$} if for every set $Z \in \Wcal$ and every (not-necessarily $Z$-computable) instance $X$ of $\Psf$, there is a solution $Y$ to $X$ such that $Z \oplus Y \in \Wcal$.
\end{definition}

Considering strong preservation has two main justifications. First, its reflects the \emph{combinatorial} weakness of problems, as opposed to the \emph{computational} weakness of standard notions of preservation. Indeed, by proving that the infinite pigeonhole principle admits strong avoidance of $\omega$ cones, Dzhafarov and Jockusch~\cite{Dzhafarov2009Ramseys} show that there is an intrinsic combinatorial weakness in the pigeonhole principle which prevents the coding of arbitrary sets in the collection of solutions. On the other hand, the proof of Seetapun~\cite{Seetapun1995strength} that Ramsey's theorem for pairs admits avoidance of $\omega$ cones strongly relies on the effectiveness of the colorings of pairs. When removing the effectiveness restriction, one can code any hyperarithmetical set, and therefore $\rt^2_2$ does not admit strong avoidance of $\omega$ cones. These results can be considered as interesting \emph{per se}.

The second reason is more technical, and specific to Ramsey-like statements. Many such statements are about colorings over $[\omega]^n$ and are parametrized by the size $n$ of the tuples. See for example Ramsey's theorem~\cite{Jockusch1972Ramseys}, the thin set~\cite{Cholak2001Free}, free set~\cite{Cholak2001Free}, and rainbow Ramsey~\cite{Wang2014Some} theorems. Such theorems admit inductive proofs based on $n$. Proofs that such a statement $\Psf^{n+1}$ admits some preservation are usually obtained by proving that $\Psf^n$ admits strong preservation of the property, and then deducing the non-strong version for $\Psf^{n+1}$. One can even obtain reversals for the notions of avoidance mentioned in this article. See for example Theorem 1.5. of~\cite{Cholak2019Thin}. 

One can directly deduce implications between strong notions of preservation from their corresponding weak notions of preservation.

\begin{theorem}
Suppose preservation of $\Wcal_1$ implies preservation of $\Wcal_2$.
Then strong preservation of $\Wcal_1$ implies strong preservation of $\Wcal_2$.
\end{theorem}
\begin{proof}
Let $\Psf$ be a problem which admits strong preservation of $\Wcal_1$. We prove that $\Psf$ admits strong preservation of $\Wcal_2$.
Fix a set $Z \in \Wcal_2$ and an instance $X$ of $\Psf$.
Let $\tilde{\Psf}$ be the problem whose unique instance is $\emptyset$, and whose solutions are the $\Psf$-solutions of $X$. In particular, $\tilde{\Psf}$ admits preservation of $\Wcal_1$, so it admits preservation of $\Wcal_2$. Let $Y$ be a $\tilde{\Psf}$-solution to the $\tilde{\Psf}$-instance $\emptyset$ such that $Z \oplus Y \in \Wcal_2$. In particular, $Y$ is a $\Psf$-solution to the $\Psf$-instance $X$.
\end{proof}

The remainder of this article is devoted to proving the equivalences and non-equivalences of the notions of preservation presented above.

\section{Avoiding cones}

The goal of this section is to prove the following theorem. The variety of notions of preservation which happen to be equivalent can be taken as an argument in favor of the naturality of the notion.

\begin{theorem}\label{thm:avoiding-cones-equivalence}
	Let $\Psf$ be a problem. Then the following are equivalent:
	\begin{enumerate}
		\item $\Psf$ admits avoidance of 1 cone.
		\item $\Psf$ admits avoidance of $\omega$ cones.
		\item $\Psf$ admits preservation of 1 non-$\Sigma^0_1$ definition.
		\item $\Psf$ admits preservation of 1 hyperimmunity.
	\end{enumerate}
\end{theorem}

The proof of \Cref{thm:avoiding-cones-equivalence} breaks into several parts, corresponding to subsections. In the first part, we prove the equivalence between avoiding 1 cone and avoiding $\omega$ cones. Then, we prove the equivalence between avoiding 1 cone and preserving 1 non-$\Sigma^0_1$ definition. Last, we prove the equivalence between avoiding 1 cone and preserving 1 hyperimmunity.

\subsection{Avoiding $\omega$ cones}

We start by proving that the notions of avoidance of 1 cone and of $\omega$ cones coincide. For this, we need to prove two lemmas which say that given a collection of non-zero Turing degrees $\mathbf{d}_0, \mathbf{d}_1, \dots$, one can always find a degree $\mathbf{e}$ relative to which these degrees collapse into a single non-zero degree. 

\begin{lemma}\label{lem:one-to-omega-cone-avoidance-1}
	Fix $Z$ and $B,A_0,A_1,A_2,\ldots \nleq_T Z$.
	Then there is $G$ such that $A_1,A_2,\ldots \nleq_T Z \oplus G$ but $A_0,A_1,\ldots \leq_T Z \oplus G \oplus B$. 	
\end{lemma}
\begin{proof}
	We may assume that $B$ is not $\Sigma^0_1(Z)$. We build $G$ by finite extensions as the union of a sequence $\sigma_{-1} \subseteq \tau_{0} \subseteq \sigma_0 \subseteq \tau_1 \subseteq \sigma_1 \subseteq \tau_1 \subseteq \cdots$.
	
	Begin $\sigma_{-1} = \varnothing$. In general, let $\sigma_{\la j,k \ra} = \tau_{\la j,k \ra} \concat \la A_j(k) \ra$. Given $\sigma_i$, define $\tau_i$ as follows. Let $i = \la j,k \ra$. Define a $\Sigma^0_1(Z)$ set $D$ such that $n \in D$ if and only if there are $\rho_0,\rho_1 \succeq \sigma_i \concat 0^n \concat 1$ and $\ell$ with
	\[ \Phi_j^{Z \oplus \rho_0}(\ell) \neq \Phi_j^{Z \oplus \rho_1}(\ell).\]
	Now as $B$ is not $\Sigma^0_1(Z)$ there is $n \in B \triangle D$. If $n \in B - D$, then let $\tau_i = \sigma_i \concat 0^n \concat 1$. If $n \in D - B$, find the first witness $(\rho_0,\rho_1,\ell)$ and choose $\tau_i$ to be whichever $\rho$ has
	\[ \Phi_j^{Z \oplus \rho}(\ell) \neq A_k(\ell).\]
	This completes the construction of $G$.
	
	First we claim that $A_k \nleq_T Z \oplus G$. Indeed, suppose that $\Phi_j^{Z \oplus G} = A_k$. Let $i = \la j,k \ra$. Then, when defining $\tau_i$, we must have had $n \in B - D$, or we would have chosen $\tau_i$ such that for some $\ell$
	\[ \Phi_j^{Z \oplus \tau_i}(\ell) \neq A_k(\ell).\]
	So we set $\tau_i = \sigma_i \concat 0^n \concat 1$ and for all $\rho_0,\rho_1 \succeq \tau_i$ and $\ell$,
	\[ \Phi_j^{Z \oplus \rho_0}(\ell)\downarrow \text{ and } \Phi_j^{Z \oplus \rho_1}(\ell) \downarrow \qquad  \Longrightarrow  \qquad \Phi_j^{Z \oplus \rho_0}(\ell) = \Phi_j^{Z \oplus \rho_1}(\ell).\]
	Thus $Z \geq_T A_k$, a contradiction.
	
	Finally, we argue that for each $j$, $A_j \leq_T Z \oplus G \oplus B$. This is because $Z \oplus G \oplus B$ can reconstruct the sequence $\sigma_{-1} \subseteq \tau_0 \subseteq \sigma_0 \subseteq \tau_1 \subseteq \cdots$, and $\sigma_{\la j,k \ra} = \tau_{\la j,k \ra} \concat \la A_j(k) \ra$. Indeed, given $\tau_i$, $G$ can determine $\sigma_{i+1}$. Given $\sigma_i$, using $G$ we can find $n$ such that $\sigma_i \concat 0^n \concat 1 \prec G$. If $n \in B$, then $\tau_i = \sigma_i \concat 0^n \concat 1$. If $n \notin B$, let $i = \la j,k \ra$ and search for the first $\rho_0,\rho_1 \succeq \sigma_i$ and $\ell$ with
	\[ \Phi_j^{Z \oplus \rho_0}(\ell) \neq \Phi_j^{Z \oplus \rho_1}(\ell).\]
	Then $\tau_i$ is whichever $\rho$ has $\rho \prec G$.
\end{proof}

\begin{lemma}\label{lem:one-to-omega-cone-avoidance-2}
	Fix $Z$ and $B, A_0,A_1,A_2,\ldots \nleq_T Z$.
	Then there is $G$ such that $B \nleq_T Z \oplus G$ but, for each $i$, $B \leq_T Z \oplus G \oplus A_i$. 	
\end{lemma}
\begin{proof}
	Using \Cref{lem:one-to-omega-cone-avoidance-1} we can inductively choose $G_0,G_1,G_2,\ldots$ such that for each $n$, $B,A_{n+1},A_{n+2},\ldots \nleq_T Z \oplus G_0 \oplus \cdots \oplus G_n$ but $B \leq_T Z \oplus G_0 \oplus \cdots \oplus G_n \oplus A_n$.
	
	We will define $H = \bigoplus H_n$ where $H_n =^* G_n$. We want to have that $B \nleq_T Z \oplus H$; since $H_n =^* G_n$, it will be automatic that for each $i$, $B \leq_T Z \oplus H \oplus A_i$. We define $H$ by forcing; our conditions are of the form $H_0 \oplus \cdots \oplus H_\ell$ where $H_n =^* G_n$. We argue that given a Turing reduction $\Phi$ and a condition $H_0 \oplus \cdots \oplus H_\ell$, there is an extension $H_0 \oplus \cdots \oplus H_k$ such that we either force that $\Phi^{Z \oplus H}$ is partial or that $\Phi^{Z \oplus H} \neq B$. If there are $x$, $k$, $\sigma_{\ell+1},\ldots,\sigma_k$ such that
	\[ \Phi^{Z \oplus H_0 \oplus \cdots \oplus H_\ell \oplus \sigma_{\ell+1} \oplus \cdots \oplus \sigma_k}(x) \downarrow \neq B(x) \]
	then we can find a condition extending $H_0 \oplus \cdots \oplus H_\ell$ which forces that $\Phi^{Z \oplus H} \neq B$. Otherwise, suppose that for each $x$ there are $\sigma_{\ell+1},\ldots,\sigma_k$ such that
	\[ \Phi^{Z \oplus H_0 \oplus \cdots \oplus H_\ell \oplus \sigma_{\ell+1} \oplus \cdots \oplus \sigma_k}(x) \downarrow. \]
	Then $B \leq_T Z \oplus H_0 \oplus \cdots \oplus H_\ell$, a contradiction. So there must be some $x$ such that for all $\sigma_{\ell+1},\ldots,\sigma_k$,
	\[ \Phi^{Z \oplus H_0 \oplus \cdots \oplus H_\ell \oplus \sigma_{\ell+1} \oplus \cdots \oplus \sigma_k}(x) \uparrow. \]
	Then $H_0 \oplus \cdots \oplus H_\ell$ already forces that $\Phi^{Z \oplus H}(x)$ does not converge.
\end{proof}

\begin{corollary}
Avoidance of 1 cone implies avoidance of $\omega$ cones.
\end{corollary}
\begin{proof}
Let $\Psf$ be a problem admitting avoidance of 1 cone.
Fix some set $Z$, non-$Z$-computable sets $A_0, A_1, \dots$ and a $Z$-computable instance $X$ of $\Psf$. By \Cref{lem:one-to-omega-cone-avoidance-2}, letting $B = A_0$, there is a set $G$ such that $B \nleq_T Z \oplus G$ but, for each $i$, $B \leq_T Z \oplus G \oplus A_i$. Since $\Psf$ admits avoidance of 1 cone, there is a $\Psf$-solution $Y$ to $X$ such that $B \nleq_T Z \oplus G \oplus Y$. We claim that for every $i \in \omega$, $A_i \nleq_T Z \oplus Y$. Indeed, otherwise, $B \leq_T Z \oplus G \oplus A_i$, but then $B \leq_T Z \oplus G \oplus Y$, contradiction.
\end{proof}

However, when considering non-relativized versions of cone avoidance, avoiding 2 cones is strictly stronger than avoiding 1 cone. We call \emph{unrelativized} a notion for which the ground set $Z$ is $\emptyset$. A pair of Turing degrees $\mathbf{a}, \mathbf{b}$ is \emph{minimal} if they are both non-zero, $\mathbf{0}$ is the only degree below both of them.

\begin{theorem}
There is a problem which admits non-relativized avoidance of 1 cone, but not of 2 cones.
\end{theorem}
\begin{proof}
Fix two sets $A$ and $B$ whose Turing degrees form a minimal pair. Let $\Psf$ be the problem with unique instance $\emptyset$. A solution is either of $A$ or $B$. $\Psf$ does not admit non-relativized avoidance of 2 cones, as witnessed by taking the cones $A$ and $B$. On the other hand, $\Psf$ admits non-relativized avoidance of 1 cone. Indeed, let $C$ be a non-computable set, and consider the unique instance $\emptyset$ of $\Psf$. By minimality of the pair of degrees of $A$ and $B$, either $A \not \geq_T C$ or $B \not \geq_T C$. In either case, there is a $\Psf$-solution $Y \in \{A,B\}$ to $\emptyset$ such that $C \not \leq_T Y$.
\end{proof}

The equivalence between the two relativized notions show in particular that there is no pair of Turing degrees which is minimal relative to every degree which lies above neither of them.

\subsection{Preserving 1 hyperimmunity}

We now prove that preserving 1 cone is equivalent to preserving 1 hyperimmunity.
The forward implication is relatively simple.

\begin{lemma}\label{lem:hyperimmune-is-modulus-of-non-zero}
Fix a set $Z$ and a nondecreasing $Z$-hyperimmune function $f : \omega \to \omega$. 
There is a set $G$ and a $\Delta^0_2(G)$ set $A \not \leq_T Z \oplus G$ such that $f$ is a $G$-modulus for $A$.
\end{lemma}
\begin{proof}
We construct $G$ which will be a $\Delta^0_2$-approximation of $A$, with $f$ a $G$-modulus for $A$. More precisely, 
\[
(\forall x)(\forall y > f(x))\, G(x, y) = G(x, f(x)) = A(x)
\]

It is now clear that for any $h$ dominating $f$, $A \leq_T G \oplus h$. It remains only to show that $A \not \leq_T Z\oplus G$.
We will construct our set $G$ by forcing. A condition is a partial function $\sigma: \omega^2 \to 2$ with finite domain. 
The function $\sigma$ is a stem for the $\Delta^0_2$ approximation $G : \omega^2 \to 2$. 
We moreover require that there can only be $(x, y), (x, z) \in \dom(\sigma)$ with $\sigma(x, z) \neq \sigma(x, y)$ if $f(x) \geq \min (y, z)$. This ensures that $f$ is a modulus for the convergence of $G$.

A condition $\tau$ \emph{extends} $\sigma$ (written $\tau \leq \sigma$) if $\tau \supseteq \sigma$. 
Every sufficiently generic filter yields $G$ such that $G$ is a stable function whose limit is $A \coloneqq \lim_s G(\cdot, s)$. We now prove that the set of conditions forcing $\Phi^{G \oplus Z}_e \neq A$ is dense.

Fix a condition $\sigma$. 
For each $x \le n$, let $s_x$ be largest with $(x, s_x) \in \dom(\sigma)$, if this exists, and $s_x = 0$ otherwise.  For every $n$, we define $h(n)$ by a $Z$-computable search.  We search for $\tau \supseteq \sigma$ such that:
\begin{itemize}
\item For all $x \le n$ and all $t, r \ge s_x$ with $(x,t), (x,r) \in \dom(\tau)$, $\tau(x,t) = \tau(x,r)$; and
\item $\Phi^{\tau \oplus Z}(n)\downarrow$.
\end{itemize}
We define $h(n) = |\tau|$ for the first such $\tau$ found, and $h(n)\uparrow$ if there is no such $\tau$.  Note that we are not restricting our search to conditions, as that would not be a $Z$-computable search. We have two cases.

Case 1: $h(n)\uparrow$ for some $n \in \omega$. 
Then let $\mu \prec \sigma$ be a condition such that for all $x \le n$ and all $t$ with $s_x < t \le f(x)$, $(x,t) \in \dom(\mu)$ and $\mu(x,t) = \sigma(x,s_x)$ if $(x,s_x) \in \dom(\sigma)$, and $\mu(x,t) = 0$ otherwise.  This condition forces $\Phi^{G\oplus Z}_e(n)\uparrow$.

Case 2: the function $h$ is total $Z$-computable. Since $f$ is $Z$-hyperimmune, there is some $n > |\sigma|$ such that $f(n) > h(n)$. Let $\tau$ witness that $h(n)\downarrow = |\tau|$. Let $\hat{\tau} \supset \tau$ be obtained by defining $\hat{\tau}(n, s) = 1-\Phi^{\tau \oplus Z}(n)$ for all $s$ with $h(n) < s \le f(n)$. Note that $\tau$ has no alternations in the columns $x < n$ that weren't present in $\sigma$, and any alternations in a column $x \ge n$ occur before $|\tau| = h(n) < f(n) \le f(x)$.  There is possibly one more alteration in $\hat{\tau}$, in column $n$, but by construction this occurs before $f(n)$.  So $\hat{\tau}$ is a valid condition extending $\sigma$.  Moreover, $\hat{\tau}$ forces $\Phi^{G \oplus Z}_e(n)\downarrow \neq A$\end{proof}

\begin{corollary}
Avoidance of 1 cone implies preservation of 1 hyperimmunity.
\end{corollary}
\begin{proof}
Suppose a problem $\Psf$ admits avoidance of 1 cone.
Fix a set $Z$, a $Z$-hyperimmune function $f : \omega \to \omega$ and a $\Psf$-instance $X \leq_T Z$. By \cref{lem:hyperimmune-is-modulus-of-non-zero}, there is a set $G$ and a $\Delta^0_2(G)$ set $A \not \leq_T Z \oplus G$ such that $f$ is a modulus for $A$. Since $\Psf$ admits avoidance of 1 cone, then there is a $\Psf$-solution $Y$ to $X$ such that $A \not \leq_T Z \oplus G \oplus Y$. If $f$ is not $Z \oplus Y$-hyperimmune, then $Z \oplus Y$ computes a function $h$ dominating $f$, and since $f$ is a $Z \oplus G$-modulus for $A$, $Z \oplus G \oplus Y$ computes $A$, contradiction.
\end{proof}

\begin{lemma}[Patey~\cite{PateyCombinatorial}]\label{lem:making-a-set-delta2}
For every set $Z$, every closed set $\Ccal \subseteq \omega^\omega$
with no $Z$-computable member, and every set $A$, there is a set $G$ such that $\Ccal$ has no $Z \oplus G$-computable member and $A$ is $\Delta^0_2(G)$.
\end{lemma}
\begin{proof}
Consider the notion of forcing whose conditions are pairs $(\sigma, n)$, where $\sigma$ is a partial function $\subseteq \omega^2 \to 2$ with finite support,  and $n \in \omega$. Informally, $\sigma$ is a stem of the $\Delta^0_2$ approximations $G : \omega^2 \to 2$ of $A$, and $n$ specifies that the $n$ first columns of $\sigma$ are already locked to $A \uh n$.
Accordingly, a condition $(\tau, m)$ \emph{extends} $(\sigma, n)$ (written $(\tau, m) \leq (\sigma, n)$) if $\tau \supseteq \sigma$, $m \geq n$, and for every $x < n$ and $t$ with $(x, t) \in \dom \tau\setminus \dom(\sigma)$, $\tau(x,t) = A(x)$. Any sufficiently generic filter yields a stable function whose limit we denote~$A$.

We now prove that the set of conditions forcing $\Phi_e^{G \oplus Z}$ not to be a member of $\Ccal$ is dense. Given a condition $(\sigma, n)$, define a $Z$-computable decreasing sequence of conditions $(\sigma, n) \geq (\tau_0, n) \geq (\tau_1, n) \geq \dots$ such that for every $i$, $\Phi_e^{\tau_i \oplus Z}(i)\downarrow$. We have two cases. In the first case, this sequence is finite, with some maximal element $(\tau_k, n)$. Then the condition $(\tau_k, n)$ is an extension of $(\sigma, n)$ forcing $\Phi^{G \oplus Z}_e(k+1)\uparrow$. In the second case, the sequence is infinite. Since $\Ccal$ has no computable member, and by closure of $\Ccal$, there must be some $k \in \omega$ such that $\Phi^{\tau_k \oplus Z}_e \uh k \downarrow = \rho$ for some $\rho \in \omega^{<\omega}$ such that $\Ccal \cap [\rho] = \emptyset$. Again, the condition $(\tau_k, n)$ is an extension of $(\sigma, n)$ forcing $\Phi^{G \oplus Z}_e$ not to be a member of $\Ccal$. This completes the proof of the lemma.
\end{proof}

\begin{lemma}\label{lem:hyperimmune-is-modulus}
	Fix a set $Z$ and $C \nleq_T Z$. There is a set $G$ and a function $f \colon \omega \to \omega$ such that $f$ is $G \oplus Z$-hyperimmune, but $Z \oplus G \oplus C$ computes a function dominating $f$.
\end{lemma}
\begin{proof}
	By Lemma \ref{lem:making-a-set-delta2} applies to $Z$ and the closed set $\{C\}$, there is $G$ such that $C$ is $\Delta^0_2(Z \oplus G)$ but $C \nleq_T Z \oplus G$. Fix a $\Delta^0_2$ approximation $C(x,s)$ for $C$ relative to $Z \oplus G$. Let $f \colon \omega \to \omega$ be the modulus for $C$ with respect to this approximation, i.e., $f(n)$ is the least $s \geq n$ such that for all $m \leq n$, $C(m,s) = C(m)$. Note that $Z \oplus G \oplus C \geq_T f$, and any function dominating $f$, together with $Z \oplus G$, computes $C$: given $g$ dominating $f$, compute $C(n)$ by finding $s^* \geq n$ such that for all $s$ with $s^* \leq s \leq g(s^*)$, $C(m,s) = C(m,s^*)$; then $C(n) = C(n,s^*)$ (see \cite{Groszek2007Moduli}).
	
	Then $f$ is $G \oplus Z$-hyperimmune, as any function dominating $f$ would together with $Z \oplus G$ compute $C$, and $C \nleq_T Z \oplus G$. But $Z \oplus G \oplus C$ computes a function dominating $f$, namely $f$ itself.
\end{proof}

\begin{corollary}
	Preservation of 1 hyperimmunity implies avoidance of 1 cone.
\end{corollary}
\begin{proof}
	Suppose that $\Psf$ admits preservation of one hyperimmunity. Fix a set $Z$, a $C \nleq_T Z$, and a $Z$-computable instance $X$ of $\Psf$. By Lemma \ref{lem:hyperimmune-is-modulus}, there is a set $G$ and a function $f \colon \omega \to \omega$ such that $f$ is $G \oplus Z$-hyperimmune but $Z \oplus G \oplus C$ computes a function dominating $f$. Since $\Psf$ admits preservation of one hyperimmunity, there is a solution $Y$ to $X$ such that $f$ is $Z \oplus G \oplus Y$-hyperimmune. Then $C \nleq_T Z \oplus Y$.
\end{proof}

Here again, we can consider unrelativized versions of these notions of preservation, and prove that they do not coincide.

\begin{theorem}
There is a problem which admits non-relativized avoidance of 1 cone, but not non-relativized preservation of 1 hyperimmunity.
\end{theorem}
\begin{proof}
Fix a $\Delta^1_1$-random set $A = \{x_0 < x_1 < \dots \}$, and let $p_A$ denote its principal function, that is, the function defined by $p_A(n) = x_n$. Let $\Psf$ be the problem with unique instance $\emptyset$. A solution is any function dominating $p_A$. Since $A$ is $\Delta^1_1$-random, it is in particular hyperimmune~\cite{Kautz1991Degrees}, so $\Psf$ does not admit unrelativized preservation of 1 hyperimmunity. We now prove that $\Psf$ admits non-relativized avoidance of 1 cone. Fix a non-computable set $C$, and the unique instance $\emptyset$ of $\Psf$. If $C$ is hyperarithmetical, then since any $\Delta^1_1$-random forms a minimal pair with any non-zero hyperarithmetical set~\cite{DoNiWe06}, $C \not \leq_T A$, and therefore $p_A$ is a $\Psf$-solution to $\emptyset$ such that $C \not \leq_T p_A$. If $C$ is non-hyperarithmetical, then it does not admit a modulus~\cite{Groszek2007Moduli}, and therefore there is a function $f : \omega \to \omega$ dominating $p_A$ such that $C \not \leq_T f$. In either case, there is a $\Psf$-solution $f$ to $\emptyset$ such that $C \not \leq_T f$.
\end{proof}

The other direction does not hold either. A Turing degree $\mathbf{d}$ is \emph{hyperimmune-free} if it does not bound a hyperimmune function. There exists non-zero hyperimmune-free degrees.

\begin{theorem}
There is a problem which admits non-relativized preservation of 1 hyperimmunity, but not non-relativized avoidance of 1 cone.
\end{theorem}
\begin{proof}
Fix a set $A$ of non-zero hyperimmune-free degree. Let $\Psf$ be the problem with unique instance $\emptyset$. The unique solution is the set $A$.
Clearly, $\Psf$ does not admit non-relativized avoidance of 1 cone.
On the other hand, $\Psf$ admits non-relativized preservation of 1 hyperimmunity. Indeed, fix a hyperimmune function $f$ and the unique $\Psf$-instance $\emptyset$. Since every $A$-computable function is dominated by a computable function, $f$ is hyperimmune relative to $A$.
\end{proof}

Again, the fact that the relativized version of these notions coincide shows that every hyperimmune function behaves, relative to some degree, like a modulus for a non-computable set. On the other hand, no non-zero hyperimmune-free degree remains hyperimmune-free relative to every degree strictly below it.

\subsection{Preserving 1 non-$\Sigma^0_1$ definition}

We now prove our last equivalence of \Cref{thm:avoiding-cones-equivalence}, namely, preserving 1 non-$\Sigma^0_1$ definition is equivalent to avoiding 1 cone. The first direction is immediate, given the fact that if a set is non-computable, then either it or its complement is not $\Sigma^0_1$.

\begin{lemma}
	Preservation of 1 non-$\Sigma^0_1$ definition implies avoidance of 1 cone.
\end{lemma}
\begin{proof}
	Suppose that $\Psf$ admits preservation of 1 non-$\Sigma^0_1$ definition. Fix $A$ and $Z$ such that $A \nleq_T Z$, and let $X$ be a $Z$-computable instance of $\Psf$. We may assume that $A$ is not $\Sigma^0_1(Z)$, otherwise we take the complement of $A$. Then there is a solution $Y$ to $X$ such that $A$ is not $\Sigma^0_1(Z \oplus Y)$, and so $A \nleq_T Z \oplus Y$.
\end{proof}

The reversal requires several lemmas which will also be useful in a latter section, when studying the hierarchy of preservation of $k$ non-$\Sigma^0_1$ definitions. In particular, these lemmas imply the non-existence of some particular enumeration degrees, namely, totally cototal degrees.

\begin{lemma}\label{lem:non_ce_permitting_in_enumerations}
For every $A \in \Delta^0_2 - \Sigma^0_1$, there is $B \le_e A$ such that:
\begin{enumerate}
\item $B \in \Delta^0_2$; and
\item $B$ is neither left c.e.\ nor right c.e.
\end{enumerate}
Here we identify the set $B$ with the real with binary representation $0.B$.
\end{lemma}

\begin{proof}
Fix a computable sequence $(A_s)_{s \in \omega}$ converging to $A$.  We simultaneously construct $B$ and an enumeration functional $\Phi$ with $B = \Phi(A)$.  
Our functional $\Phi$ will have the property that for any $x$, there will be at most one axiom $(x, F) \in \Phi$ with $F \neq \emptyset$; from this it follows that $B \in \Delta^0_2$ (using $A \in \Delta^0_2$). 

We interpret c.e.\ sets as subsets of the rationals.  We have the following requirements to meet, for all $e \in \omega$:
\begin{itemize}
\item[$R_e$:] $B \neq \sup (W_e)$;
\item[$Q_e$:] $B \neq \inf (W_e)$.
\end{itemize}

Our construction will be a finite injury construction, and so a strategy for a given requirement will act under the assumption that no higher priority strategy will ever act.

\smallskip

{\em Strategy for requirement $Q_e$:}

\begin{enumerate}
\item Choose a large $x$, and keep both $x$ and $x+1$ out of $B$ (no axioms for $x$ or $x+1$ are to be enumerated into $\Phi$);
\item Wait for a stage $s$ and a $y \in W_{e,s}$ with $y - B_s < 2^{-(x+2)}$.
\item Declare $x \in B$ and enumerate the axiom $(x, \emptyset)$ into $\Phi$.
\end{enumerate}

\smallskip

{\em Strategy for requirement $R_e$:}

We construct a c.e.\ set $D_e$ as we work.  This set is reset whenever the strategy is initialized.

Our strategy will have modules for each $k \in \omega$.  We will begin by running the $0$-module.  For each $k$, the $k$-module may run the $(k+1)$-module, but we will argue that this iteration will eventually terminate.

Here is the $k$-module:
\begin{enumerate}
\item Choose a large $x_k$.  Let the current stage be $s_k$.  Declare $x, x+1 \in B$, enumerating the axioms $(x, F_k)$ and $(x+1, F_k)$ into $\Phi$, where $F_k$ is the positive information from $A_{s_k}\uhr{k}$.
\item Wait for a stage $s$ at which one of the following happens:
\begin{enumerate}
\item $A_s\uhr{k} \neq A_{s_k}\uhr{k}$.  In this case, enumerate axioms $(x, \emptyset)$ and $(x+1, \emptyset)$ into $\Phi$, declaring $x, x+1 \in B$.  Return to Step (1).
\item There is a $y \in W_{e,s}$ with $B_s - y < 2^{-(x+2)}$.  In this case, enumerate all of $F_k$ into $D_e$ and proceed to Step (3).
\end{enumerate}
\item Wait for a stage $s$ with $F_k \not \subseteq A_s$.  While waiting, run the $(k+1)$-module.
\item Freeze the action of any running $j$-modules for $j > k$.
\item Wait for a stage $s$ with $F_k \subseteq A_s$.  When found, return to Step (3), resuming the action of any frozen $j$-modules.
\end{enumerate}

\smallskip

{\em Full construction:} 
Whenever a strategy moves between steps, we initialize all lower priority strategies.  When a $Q_e$-strategy is initialized, we enumerate $(x, \emptyset)$ and $(x+1, \emptyset)$ into $\Phi$ for the strategy's chosen $x$, declaring them both to be in $B$.  Similarly, when an $R_e$-strategy is initialized, we enumerate $(x_j, \emptyset)$ and $(x_j+1, \emptyset)$ into $\Phi$ for all appropriate $j$.

\smallskip

{\em Verification:}

\begin{claim}
For each $e$:
\begin{enumerate}[(a)]
\item The $Q_e$-strategy eventually waits forever at Step (2) or Step (3).
\item There is some $k$ such that for all $j < k$, the $j$-module of the $R_e$-strategy eventually waits forever at Step (3), and the $k$-module eventually waits forever at Step (2) or Step (5).
\end{enumerate}

\end{claim}

\begin{proof}
By simultaneous induction.  (a)$_e$ is immediate.

For (b)$_e$, by induction there is a final time when the $R_e$-strategy is intialized.  Let $s_k$ be the stage at which $x_k$ is chosen after this final initialization, and suppose towards contradiction that $s_k$ is defined for all $k < \omega$.  Since $(A_s)_{s \in \omega}$ converges to $A$, the $k$-module cannot move between Steps (3) and (5) infinitely often, so it must be that each $k$-module eventually waits forever at Step (3).

But then each $F_k \subseteq A$, and $D_e = \bigcup_k F_k$.  Further, for any $z \in A$, there is a sufficiently large $k$ such that $z \in A_{s_k} = F_k$, so $D = A$, contrary to $A$ not being c.e.
\end{proof}

It follows that each strategy is initialized only finitely many times.

\begin{claim}
Each $Q_e$-strategy meets its requirement.
\end{claim}

\begin{proof}
Let $t$ be the final stage at which the $Q_e$-strategy is initialized, so no higher priority strategy acts at any stage $s > t$.  Consider the $x$ chosen by this strategy.  Since lower priority strategies choose their elements large, $B_s\uhr{x} = B_t\uhr{x}$ for all $s > t$.  Observe that by construction, $x+1 \not \in B$.

If the strategy waits forever at Step (2), then certainly $B \neq \inf(W_e)$.

Suppose the strategy moves from Step (2) to Step (3) at stage $t_1 > t$.  Since $x, x+1 \not \in B_{t_1}$, but $x \in B$, we have
\begin{align*}
\inf(W_e) &< B_{t_1} + 2^{-(x+2)}\\
&\le (B\uhr{x})\cat 0^\infty + 2^{-(x+2)} + 2^{-(x+2)}\\
&= (B\uhr{x})\cat 1\cat 0^\infty - 2^{-(x+1)} + 2^{-(x+2)} + 2^{-(x+2)}\\
&= (B\uhr{x})\cat1\cat0^\infty\\
&\le B.
\end{align*}
Thus $B \neq \inf(W_e)$.
\end{proof}

\begin{claim}
Each $R_e$-strategy meets its requirement.
\end{claim}

\begin{proof}
Let $k$ be such that the $k$-module of the strategy eventually waits forever at Step (2) or Step (5), and let $s_j$ for $j \le k$ be the stages at which the $x_j$ is chosen after the strategy's final initialization.  If it is defined, let $s_{k+1}$ be the same for $x_{k+1}$.  Note that for $j < k$, $s_{j+1}$ is also the stage at which the $j$-module first reaches Step (3).  Similarly, $s_{k+1}$ is defined precisely if the $k$-module reaches Step (3), in which case $s_{k+1}$ is the first stage at which this happens.

By construction, for any $y \in (x_{j}+1, x_{j+1})$, one of the following must occur:
\begin{enumerate}[(i)]
\item $(y, \emptyset)$ has been enumerated into $\Phi$ by stage $s_{j+1}$; or
\item For all finite sets $F$, $(y, F) \not \in \Phi$.
\end{enumerate}
To see this: since $x_j$ is chosen large and no higher priority strategy acts after stage $s_0$, no such $y$ can be chosen by a higher priority strategy.  Also, no such $y$ can be chosen by a lower priority strategy after stage $s_{j+1}$, since elements are always chosen large.  If $y$ is chosen by a lower priority strategy before stage $s_{j+1}$, then at stage $s_{j+1}$ we initialize that strategy and enumerate $(y,\emptyset)$ into $\Phi$, if we have not already done so.  If $y$ is chosen by no strategy, then no axioms for $y$ are ever enumerated into $\Phi$.

Note also that for all $j < k$, $x_j \in B$, and indeed $x_j \in B_s$ for any stage at which the $k$-module is running (not frozen).  So $B_{s_k}\uhr{x_k} = B_s\uhr{x_k}$ for any $s > s_k$ at which the $k$-module is not frozen.  The argument is now identical to the argument for the $Q_e$-strategy.
\end{proof}

This completes the proof.
\end{proof}

A set $X$ is \emph{semi-computable} if there is a total computable function $g : [\omega]^2 \to \omega$ such that for every $\{x,y\} \in [\omega]^2$, if $\{x,y\} \cap X \neq \emptyset$ then $g(\{x,y\}) \in \{x,y\} \cap X$.

\begin{corollary}\label{cor:enumerate-semi-computable}
For any set $A \in \Delta^0_2 - \Sigma^0_1$, there is $C \le_e A$ which is semi-computable but not $\Sigma^0_1$.
\end{corollary}

The following argument is due to Jockusch.

\begin{proof}
Fix $B \le_e A$ as from \Cref{lem:non_ce_permitting_in_enumerations}.  Since $B$ is $\Delta^0_2$, fix a computable sequence of rationals $(q_e)_{e \in \omega}$ converging to $B$.  Let $C = \{ i : q_i < B\} = \{ i : q_i \le B\}$ (since $B$ is noncomputable).  Then $C \le_e B$, and $C$ is semi-computable via the induced ordering from $\Q$.  $C$ is infinite because $B$ is not right c.e., and it is not $\Sigma^0_1$ (indeed, it is immune) because $B$ is not left c.e.
\end{proof}

A set $B$ is \emph{cototal} if $B \leq_e \overline{B}$.
An enumeration degree $\mathbf{d}$ is \emph{totally cototal} if for it contains a set $A$ such that for every $B \leq_e A$, 
$B$ is cototal (as a set).

\begin{lemma}[Arslanov, Cooper, Kalimullin \cite{Arslanov2003Splitting}]\label{lem:semi-rec-properties}
	Let $A$ be a semi-computable set. Then:
	\begin{enumerate}
		\item $A$ and $\overline{A}$ form a minimal pair in the enumeration degrees.
		\item $A$ is not cototal unless $A$ is c.e.
	\end{enumerate}
\end{lemma}
\begin{proof}
	For (1), suppose that $B \leq_e A$ and $B \leq_e \overline{A}$ via enumeration operators $\Phi$ and $\Psi$ respectively. Let $f$ be the function that witnesses that $A$ is semi-computable. Then to see that $B$ is c.e., note that
	\[ B = \{x : \text{$\exists$ finite sets $F,G$ such that $x \in \Phi^F$, $x \in \Psi^G$, and for all $a \in F$ and $b \in G$, $f(a,b) = a$}\}. \]
	
	For (2), suppose that $A$ is cototal. Then $\overline{A} \geq_e A$, and so since $A$ and $\overline{A}$ form a minimal pair in the enumeration degrees, $A \equiv_e \varnothing$.
\end{proof}

The following argument is due to Mariya Soskova (private communication).

\begin{corollary}\label{cor:no-totally-cototal}
There is no totally cototal degree above $\mathbf{0}_e$.
\end{corollary}
\begin{proof}
Suppose $A$ is a totally cototal set of non-zero degree. First we argue that $A$ is $\Delta^0_2$. Let $L_{A}$ be the set of all finite binary strings lexicographically to the left of or along $A$. Then $L_{A} \leq_e A$. Moreover, $L_{A}$ is semi-computable: let $f(x,y)$ be the left-most of $x$ and $y$. Since $L_{A} \leq_e A$, it is cototal, but by \Cref{lem:semi-rec-properties} $L_{A}$ cannot be cototal unless it is c.e. We also have that $L_A \geq_T A$ so $A$ is $\Delta^0_2$. 

Since $A$ is $\Delta^0_2$ and not $\Sigma^0_1$, by \Cref{cor:enumerate-semi-computable} there is $C \leq_e A$ which is semi-computable but not $\Sigma^0_1$. By assumption, $C$ must be cototal, and so $C \leq_e \overline{C}$. But as mentioned above, each semi-computable set forms a minimal pair in the enumeration degrees with its complement. This gives a contradiction.
\end{proof}

\begin{corollary}
Avoidance of 1 cone implies preservation of 1 non-$\Sigma^0_1$ definition.
\end{corollary}
\begin{proof}
Suppose a problem $\Psf$ admits avoidance of 1 cone.
Fix a set $Z$, a non-$\Sigma^0_1(Z)$ set $A$ and a $Z$-computable instance $X$ of $\Psf$. By \Cref{cor:no-totally-cototal} relativized to $Z$, there is a non-$\Sigma^0_1(Z)$ set $C \leq_e A \oplus Z \oplus \overline{Z}$ which is not $Z$-cototal. In other words, there is an enumeration $G$ of $\overline{C}$ such that $C$ is not $\Sigma^0_1(Z \oplus G)$. Since $\Psf$ admits avoidance of 1 cone, there is a $\Psf$-solution $Y$ to $X$ such that $C \not \leq_T Z \oplus G \oplus Y$. We claim that $A$ is not $\Sigma^0_1(Z \oplus Y)$. Indeed, otherwise $C$ would be $\Sigma^0_1(Z \oplus Y)$, and therefore $C \leq_T Z \oplus G \oplus Y$, contradiction.
\end{proof}

Note that the implication from preservation of 1 non-$\Sigma^0_1$ definition to avoidance of 1 cone is natural enough to hold again when considering their non-relativized counterparts. However, these notions are not equivalent.

\begin{lemma}[Folklore]\label{lem:equiv-semi-computable}
Let $A$ be a semi-computable set. The following are equivalent:
\begin{enumerate}[(a)]
	\item $A$ is immune
	\item $A$ is hyperimmune
\end{enumerate}
Either implies (c) that $A$ is not c.e.
\end{lemma}
\begin{proof}
By \cite[Theorem 4.1]{Jockusch1968Semirecursive}, any semi-computable set $A$ is the initial segment of a computable linear order $\Lcal$.
$(b) \rightarrow (a)$ is immediate as every hyperimmune set is immune. 
$(a) \rightarrow (c)$ is also immediate as every infinite c.e.\ set contains an infinite computable subset.
Last, we prove $(a) \rightarrow (b)$. Suppose $A$ is not hyperimmune. Then there is a computable strong array $F_0, F_1, \dots$ such that for every $n \in \omega$, $F_n \cap A \neq \emptyset$. Then $\{ \min_{\Lcal} F_n : n \in \omega \}$ is an infinite c.e.\ subset of $A$ and contains an infinite computable subset.
\end{proof}

\begin{theorem}
There is a problem which admits non-relativized avoidance of 1 cone, but not non-relativized preservation of 1 non-$\Sigma^0_1$ definition.
\end{theorem}
\begin{proof}
Fix a computable linear ordering $\Lcal$ of order type $\omega+\omega^{*}$ with no infinite computable ascending or descending sequence. Such a linear order exists by Tennenbaum (see~\cite{Rosenstein1982Linear}). Let $A$ be the $\omega$ part of this linear order. In particular, $A$ and $\overline{A}$ are both $\Delta^0_2$, semi-computable and immune. By \Cref{lem:equiv-semi-computable}, $A$ and $\overline{A}$ are both non-$\Sigma^0_1$ and hyperimmune.
Let $\Psf$ be the problem with unique instance $\emptyset$. A solution is an infinite subset of~$A$.

For any solution $Y \subseteq A$, $A = \{ x : \exists y \in Y\, x \le_{\Lcal} y\}$, so $A \in \Sigma^0_1(Y)$, and thus 
$\Psf$ does not admit non-relativized preservation of 1 non-$\Sigma^0_1$ definition.
We claim that $\Psf$ admits non-relativized avoidance of 1 cone.
Fix a non-computable set $C$ and the unique $\Psf$-instance~$\emptyset$.
If $C$ is not $\Delta^0_2$, then $A$ is a $\Psf$-solution to $\emptyset$ such that $C \not \leq_T A$. If $C$ is $\Delta^0_2$, then since $\overline{A}$ is hyperimmune and $\Delta^0_2$, by Proposition 4.4 of \cite{Hirschfeldt2008strength}, there is an infinite subset $H \subseteq A$ such that $C \not \leq_T H$. In both cases, there is a $\Qsf$-solution $Y$ to $\emptyset$ such that $C \not \leq_T Y$.
\end{proof}

\section{The hierarchy of preservations}

Given a coloring $f : [\omega]^n \to k$, an infinite set $H \subseteq \omega$ is \emph{$f$-homogeneous} if $f$ uses only one color on $[H]^n$. Ramsey's theorem asserts the existence of homogeneous sets for every $k$-coloring of $[\omega]^n$. Jockusch~\cite{Jockusch1972Ramseys} proved that whenever $n \geq 3$, there is a computable coloring $f : [\omega]^n \to 2$ such that every $f$-homogeneous set computes $\emptyset'$. However, Wang~\cite{Wang2014Some} proved the surprising result that this is no longer the case when we relax the $f$-homogeneity condition to allow more colors. 
	
\begin{definition}
For every $n, \ell \geq 2$, let $\rt^n_{<\infty, \ell}$ be the problem whose instances are functions $f : [\omega]^n \to k$ for some $k \in \omega$. An $\rt^n_{<\infty, \ell}$-solution to $f$ is an infinite set $H \subseteq \omega$ such that $|f[H]^n| \leq \ell$.
\end{definition}

Wang~\cite{Wang2014Some} proved that for every $n \geq 1$, there is some $\ell$ such that $\rt^n_{<\infty, \ell}$ admits cone avoidance.  Patey~\cite{Patey2016weakness,Patey2017Iterative} proved the following theorem, which shows in particular that the hierarchies of preservation of $\ell$ hyperimmunities and $\ell$ non-$\Sigma^0_1$ definitions is strictly increasing.

\begin{theorem} For every $\ell \geq 1$, $\rt^2_{<\infty, \ell}$ admits preservation of $\ell$ but not $\ell+1$ non-$\Sigma^0_1$ definitions, and of $\ell$ but not $\ell+1$ hyperimmunities.	
\end{theorem}

Let us sketch the proof that $\rt^n_{<\infty, \ell}$ does not admit preservation of $\ell+1$ hyperimmunities. Given $\ell \geq 1$, build a $\Delta^0_2$ $(\ell+1)$-partition $A_0 \sqcup \dots \sqcup A_\ell = \omega$ such that for every $i \le \ell$, $\overline{A}_i$ is hyperimmune. By Schoenfield's limit lemma, there is a computable coloring $f : [\omega]^2 \to \ell+1$ such that for every $x \in \omega$, $\lim_y f(x, y)$ exists, and $x \in A_{\lim_y f(x, y)}$. We claim that for every $\rt^2_{<\infty, \ell}$-solution $H$ to $f$, there is some $i \leq \ell$ such that $\overline{A}_i$ is not $H$-hyperimmune. Since $|f[H]^2| \leq \ell$, there is some $i \leq \ell$ such that $i \not \in f[H]^2$. In particular, $H \subseteq \overline{A}_i$, so the principal function of $H$ dominates the principal function of $\overline{A}_i$, which proves that $\overline{A}_i$ is not $H$-hyperimmune.

\subsection{The hyperimmunities and non-$\Sigma^0_1$ definitions hierarchies}

We now prove that the two hierarchies of preservation of hyperimmunities and of non-$\Sigma^0_1$ definitions coincide.

\begin{lemma}\label{lem:k-non-ce-to-k-hyperimmune-preservation-1}
For every $k \le \omega$ and every $Z$, for any nondecreasing functions $(f_i)_{i < k}$ which are not $Z$-computably dominated, there is a $G$ and sets $(A_i)_{i < k}$ such that none of the $A_i$ is c.e.\ relative to $Z\oplus G$, but for any $i$ and any function $h$ dominating $f_i$, $A_i$ is c.e.\ relative to $Z\oplus G \oplus h$.
\end{lemma}

\begin{proof}
We construct $(G_i)_{i < k}$ which will be $\Pi^0_2$-approximations to the $A_i$, with each $f_i$ a $\Sigma^0_1$-modulus for $G_i$.  That is,
\[
x \in A_i \iff \exists^\infty s\, [(x,s) \in G_i] \iff \exists s > f_i(x)\, [(x,s) \in G_i].
\]
Then $G = \bigoplus_{i < k} G_i$.  It is now clear that for any $h$ dominating $f_i$, $A_i$ is c.e.\ relative to $Z\oplus G \oplus h$.  It remains only to show that none of the $A_i$ is c.e.\ relative to $Z\oplus G$.

We will construct our $G_i$ generically.  Conditions in our notion of forcing are pairs of sequences $( (\sigma_i)_{i < k}, (N_i)_{i < k})$, with:
\begin{itemize}
\item $\sigma_i \in 2^{<\omega\times \omega}$;
\item $N_i \in [\omega]^{<\omega}$;
\item If $x \in N_i$, $s > f_i(x)$ and $(x, s) \in \dom(\sigma_i)$, then $\sigma_i(x,s) = 0$; and
\item All but finitely many of the $\sigma_i$ and $N_i$ are empty.
\end{itemize}
Of course the last requirement only matters when $k = \omega$.  Extension is defined elementwise.

Note that for a sufficiently generic filter $F$, $x \in A_i^F \iff x \not \in N_i^F$.

Note also that for any $s > f_j(x)$ and any condition $\rho = ( (\sigma_i)_{i < k}, (N_i)_{i < k})$, if $\sigma_j(x,s) = 1$ then $\rho \Vdash [x \not \in N_i]$.  So for a sufficiently generic filter $F$, each $A_i^F$ is infinite.

For a c.e.\ operator $W$ and a $j < k$, we must show that for a sufficiently generic $G$, $A_j \neq W^{Z\oplus G}$.  Given a condition $\rho = ((\sigma_i)_{i < k}, (N_i)_{i < k})$, we define a function $g$.  For each $n$, search $Z$-effectively for a $(\tau_i)_{i< k}$ and a $y \in \omega$ such that:
\begin{itemize}
\item $y > n$;
\item For each $i < k$, $\tau_i$ extends $\sigma_i$;
\item For all $i < k$, $x \in N_i$ and $s > f_i(x)$, we have $\tau_i(x,s) \neq 1$; and
\item $y \in W^{Z\oplus (\tau_i)_{i < k}}$.
\end{itemize}
Note that this search is $Z$-effective, albeit nonuniformly in the information $\{(i, x, f_i(x)) : x \in N_i\}$.
For the first $y$ and $(\tau_i)_{i < k}$ found, define $g(n)$ to be the largest $s$ with $\tau_j(y,s) = 1$, or $g(n) = 0$ if no such $s$ exists.

If some $g(n)$ is undefined, then no extension of $\rho$ forces any $y > n$ into $W^{Z\oplus G}$, and so $W^{Z\oplus G}$ is finite.  But we already said that $A_j$ is infinite, and so $\rho \Vdash [A_j \neq W_j^{Z\oplus G}]$.

If $g$ is total, then since $g$ is $Z$-computable, there must be an $n$ with $g(n) < f_j(n)$.  Let $(\tau_i)_{i < k}$ and $y$ be the witnesses to the definition of $g(n)$.  Define $M_i = N_i$ for $i \neq j$, and define $M_j = N_j \cup \{ y \}$.  Since $f_j(n) \le f_j(y)$, there is no $s > f_j(y)$ with $\tau_j(y, s) = 1$, so $\hat{\rho} = ( (\tau_i)_{i < k}, (M_i)_{i < k})$ is a condition extending $\rho$, and $\hat{\rho} \Vdash [y \in W^{Z\oplus G} - A_j]$.
\end{proof}

\begin{corollary}
For any $k \le \omega$, preservation of $k$ non-$\Sigma^0_1$ definitions implies preservation of $k$ hyperimmunities.
\end{corollary}
\begin{proof}
Suppose a problem $\Psf$ admits preservation of $k$ non-$\Sigma^0_1$ definitions. Fix a set $Z$, $k$ $Z$-hyperimmune functions $f_0, f_1, \dots, f_{k-1}$ and a $Z$-computable $\Psf$-instance $X$. By~\cref{lem:k-non-ce-to-k-hyperimmune-preservation-1}, there is a $G$ and sets $(A_i)_{i < k}$ such that none of the $A_i$ is $\Sigma^0_1(Z\oplus G)$, but for any $i$ and any function $h$ dominating $f_i$, $A_i$ is $\Sigma^0_1(Z\oplus G \oplus h)$. Since $\Psf$ admits preservation of $k$ non-$\Sigma^0_1$ definitions, there is a $\Psf$-solution $Y$ to $X$ such that for every $i < k$, $A_i$ is not $\Sigma^0_1(Z \oplus G \oplus Y)$. In particular, for every $i < k$, $f_i$ is $Z \oplus G \oplus Y$-hyperimmune.
\end{proof}

\begin{lemma}\label{lem:making-a-set-delta2-and-preserving-nonce}
For every set $Z$, every countable sequence of non-$Z$-c.e.\ sets $B_0, B_1, \dots$, and every set $A$, there is a set $G$ such that $B_i$ is not $Z \oplus G$-c.e.\ for every $i \in \omega$ and $A$ is $\Delta^0_2(G)$.
\end{lemma}
\begin{proof}
Consider again the notion of forcing whose conditions are pairs $(\sigma, n)$, where $\sigma$ is a partial function $\subseteq \omega^2 \to 2$ with finite support,  and $n \in \omega$. A condition $(\tau, m)$ \emph{extends} $(\sigma, n)$ if $\tau \supseteq \sigma$, $m \geq n$, and for every $(x, y) \in \dom \tau \setminus \dom \sigma$ such that $x < n$, $\tau(x, y) = A(x)$. Any sufficiently generic filter yields a stable function whose limit is~$A$.

We now prove that the set of conditions forcing $W_e^{G \oplus Z} \neq B_i$ is dense. Given a condition $(\sigma, n)$, let $U = \{ x : \exists (\tau, n) \leq (\sigma, n)\, x \in W_e^{\tau \oplus Z}\}$. The set $U$ is $\Sigma^0_1(Z)$, so $B_i \Delta U \neq \emptyset$. If there is some $x \in B_i \setminus U$ then the condition $(\sigma, n)$ already forces $x \not \in W_e^{G \oplus Z}$ are we are done. If there is some $x \in U \setminus B_i$, then then condition $(\tau, n) \leq (\sigma, n)$ such that $x \in W_e^{\tau \oplus Z}$ is an extension forcing $x \in W_e^{G \oplus Z}$. In both cases, there is an extension forcing $W_e^{G \oplus Z} \neq B_i$. This completes the proof of the lemma.
\end{proof}

\begin{corollary}
For any $k \leq \omega$, preservation of $k$ hyperimmunities implies preservation of $k$ non-$\Sigma^0_1$ definitions.
\end{corollary}
\begin{proof}
Suppose some problem $\Psf$ admits preservation of $k$ hyperimmunites.
Fix a set $Z$ and $k$ non-$\Sigma^0_1(Z)$ sets $A_0, \dots, A_{k-1}$.
By \Cref{lem:making-a-set-delta2-and-preserving-nonce}, there is a set $G$ such that $A_0, \dots, A_{k-1}$ are not $\Sigma^0_1(Z \oplus G)$, but $A_0 \oplus \dots \oplus A_{k-1}$ is $\Delta^0_2(G)$. By \Cref{cor:enumerate-semi-computable}, there are semi-$Z$-computable sets $B_0 \leq_e A_0 \oplus Z \oplus \overline{Z}, \dots, B_{k-1} \leq_e A_{k-1} \oplus Z \oplus \overline{Z}$ such that for every $i < k$, $B_i$ is not $\Sigma^0_1(Z \oplus G)$. By \Cref{lem:equiv-semi-computable}, $B_i$ is $Z \oplus G$-hyperimmune. Since $\Psf$ admits preservation of $k$ hyperimmunites, there is a $\Psf$-solution $Y$ to $X$ such that $B_i$ is $Z \oplus G \oplus Y$-hyperimmune for every $i < k$. We claim that $A_i$ is not $\Sigma^0_1(Z \oplus G \oplus Y)$. Indeed, otherwise, $B_i$ would be $\Sigma^0_1(Z \oplus G \oplus Y)$, and by \Cref{lem:equiv-semi-computable}, $B_i$ would not be $Z \oplus G \oplus Y$-hyperimmune.
\end{proof}

\subsection{The hierarchy of constant-bound traces of closed sets}

One can define a similar hierarchy for avoidance of constant-bound traces of closed sets. By the hyperimmune-free basis theorem, $\wkl$ admits preservation of $\omega$ hyperimmunities (hence of $\omega$ non-$\Sigma^0_1$ definitions as well). On the other hand, $\wkl$ does not admit avoidance of constant-bound traces of even 1 closed set. Indeed, letting $\Ccal$ be the effectively closed set of all the completions of Peano arithmetics, every constant-bound trace of $\Ccal$ computes a member of $\Ccal$.
This separates the hierarchies of preservation of hyperimmunities and non-$\Sigma^0_1$ definitions from the hierarchy of constant-bound traces of closed sets.

On the other direction, avoidance of constant-bound traces for closed sets does not imply the preservation of more hyperimmunities than closed sets, as shows the following theorem.

\begin{theorem}
For every $\ell \geq 1$, there is a problem which admits preservation of constant-bound traces for $\ell$ closed sets, but not preservation of $\ell+1$ hyperimmunities.
\end{theorem}
\begin{proof}
Patey~\cite{PateyCombinatorial} proved that $\rt^2_{<\infty, \ell}$ admits avoidance of constant-bound traces of $\ell$ closed sets. On the other hand, we already argued that $\rt^2_{<\infty, \ell}$ does not admit preservation of $\ell+1$ hyperimmunities.
\end{proof}

We finish this section by proving that preservation of constant-bound traces for $k$ closed sets implies preservation of $k$ hyperimmunities.

\begin{lemma}\label{lem:ktraces-to-khyperimmunities}
For any $k \le \omega$, for any $Z \in 2^\omega$ and $(f_i)_{i < k}$ such that each $f_i$ is $Z$-hyperimmune, there is a $G \in 2^\omega$ and closed sets $(C_i)_{i < k}$ in Cantor space such that each $C_i$ has no constant-bound $(Z\oplus G)$-trace, but for any function $h$ dominating $f_i$, there is a $(Z\oplus G\oplus h)$-computable element of $C_i$.
\end{lemma}

\begin{proof}

We construct sets $(E_i)_{i < k}$ and $(F_i)_{i < k}$, which will be $\Sigma^0_2$ approximations to sets $A_i$ and $B_i$, respectively, and such that each $f_i$ is a $\Pi^0_1$-modulus for $A_i$ and $B_i$.  That is,
\[
x \in A_i \iff \exists s\, \forall t > s\, [(x, t) \in E_i] \iff \forall t > f_i(x)\, [(x,t) \in E_i]
\]
and
\[
x \in B_i \iff \exists s\, \forall t > s\, [(x, t) \in F_i] \iff \forall t > f_i(x)\, [(x,t) \in F_i].
\]
Then $G = \bigoplus_{i < k} E_i \oplus \bigoplus_{i < k} F_i$.

Our sets will have the property that $A_i \cap B_i = \emptyset$.  Each $C_i$ will then be the set of separators of $A_i$ and $B_i$.  That is, $C_i = \{ X \in 2^\omega : A_i \subseteq X \wedge B_i \subseteq \overline{X}\}$.  Observe that if $h$ dominates $f_i$, then $A_i$ and $B_i$ are $\Pi^0_1(h\oplus G)$, and so $h\oplus G$ computes an element of $C_i$.  It remains only to show that none of the $C_i$ has a constant-bound trace relative to $Z\oplus G$.

We will construct our $E_i$ and $F_i$ simultaneously generically.  Conditions in our notion of forcing are tuples of sequences $( (\sigma_i)_{i < k}, (N_i)_{i < k}, (\tau_i)_{i < k}, (M_i)_{i < k})$ with:
\begin{itemize}
\item $\sigma_i, \tau_i \in 2^{<\omega\times\omega}$;
\item $N_i, M_i \in [\omega]^{<\omega}$;
\item If $x \in N_i$, $s > f_i(x)$ and $(x,s) \in \dom(\sigma_i)$, then $\sigma_i(x,s) = 1$;
\item If $x \in M_i$, $s > f_i(x)$ and $(x,s) \in \dom(\tau_i)$, then $\tau_i(x,s) = 1$;
\item $N_i \cap M_i = \emptyset$; and
\item All but finitely many of the $\sigma_i, \tau_i, N_i$ and $M_i$ are empty.
\end{itemize}
Extension is defined elementwise.

Note that for a sufficiently generic filter, $A_i = N_i$ and $B_i = M_i$.  

For a $b \in \omega$, a $j < k$ and a Turing functional $\Phi$, we must show that for a sufficiently generically chosen $G$, $\Phi^{Z\oplus G}$ is not a $b$-bounded trace of $C_j$.  We may assume that for all oracles $X$ and all $n \in \omega$, $\Phi^X(n)$ is a subset of $2^n$ of size at most $b$.  Given a condition $\rho = ( (\sigma_i)_{i < k}, (N_i)_{i < k}, (\tau_i)_{i < k}, (M_i)_{i < k})$, we define a function $g$.  On input $n$, we search for $(\hat{\sigma}_i)_{i < k}$ and $(\hat{\tau}_i)_{i < k}$
 such that:
\begin{itemize}
\item For each $i$, $\hat{\sigma}_i$ extends $\sigma_i$, and $\hat{\tau}_i$ extends $\tau_i$;
\item For all $i < k$, $x \in N_i$ and $s > f_i(x)$ with $(x,s) \in \dom(\hat{\sigma}_i)$, we have $\hat{\sigma}_i(x,s) = 1$;
\item For all $i < k$, $x \in M_i$ and $s > f_i(x)$ with $(x,s) \in \dom(\hat{\tau}_i)$, we have $\hat{\tau}_i(x,s) = 1$; and
\item $\Phi^{Z\oplus (\hat{\sigma}_i)_{i < k} \oplus (\hat{\tau}_i)_{i < k}}(n+b)\converge$.
\end{itemize}
Note that this search is $Z$-effective, albeit nonuniformly in the information $\{(i, x, f_i(x)) : x \in N_i \cup M_i\}$.  For the first $(\hat{\sigma}_i)_{i < k}$ and $(\hat{\tau}_i)_{i < k}$ found, define $g(n)$ to be the largest $s$ with $\hat{\sigma}_j(n+a, s) = 0$ or $\hat{\tau}_j(n+a,s) = 0$ for some $a < b$, or $g(n) = 0$ if no such $s$ exists.

If some $g(n)$ is undefined, then no extension of $\rho$ forces $\Phi^{Z\oplus G}(n+b)\converge$, and so $\rho$ forces that $\Phi^{Z\oplus G}$ is partial, and thus not a $b$-bounded trace.

If $g$ is total, then since it is $Z$-computable, there must be an $n$ with $g(n) < f_j(n)$.  Let $(\hat{\sigma}_i)_{i < k}$ and $(\hat{\tau}_i)_{i < k}$ be the witnesses to the definition of $g(n)$.  Let $\Phi^{Z\oplus (\hat{\sigma}_i)_{i < k} \oplus (\hat{\tau}_i)_{i < k}}(n+b) = \{ \pi_0, \dots, \pi_{b-1}\}$.  Define $\hat{N}_i = N_i$ and $\hat{M}_i = M_i$ for $i \neq j$.  Define $\hat{N}_j = N_j \cup \{ a < b : \pi_a(n+a) = 0 \}$ and $\hat{M}_j = M_j \cup \{ a < b : \pi_a(n+a) = 1\}$.  Since $f_j(n) > g(n)$, there is no $s > f_j$ and $a < b$ with $\hat{\sigma}_j(n+a,s) = 0$ or $\hat{\tau}_j(n+a,s) = 0$, so $\hat{\rho} = ( (\hat{\sigma}_i)_{i < k}, (\hat{N}_i)_{i < k}, (\hat{\tau}_i)_{i < k}, (\hat{M}_i)_{i < k})$ is a condition extending $\rho$, and $\hat{\rho}$ forces that no element of $\Phi^{Z\oplus G}(n)$ is extendible to an element of $C_j$.
\end{proof}

\begin{corollary}
For all $k \le \omega$, avoidance of constant-bound traces for $k$ closed sets implies preserving $k$ hyperimmunities.
\end{corollary}
\begin{proof}
Suppose a problem $\Psf$ admits avoidance of constant-bound traces for $k$ closed sets. Fix a set $Z$, a collection of $Z$-hyperimmune functions $(f_i)_{i < k}$ and a $Z$-computable $\Psf$-instance $X$. By~\Cref{lem:ktraces-to-khyperimmunities}, there is a $G$ and closed sets $(\Ccal_i)_{i < k}$ in the Cantor space such that none of the $\Ccal_i$ has a constant-bound $Z \oplus G$-trace, but for any function $h$ dominating $f_i$, there is a $Z \oplus G \oplus h$-computable element of $\Ccal_i$. Since $\Psf$ admits avoidance of constant-bound traces for $k$ closed sets, there is a $\Psf$-solution $Y$ to $X$ such that for every $i < k$, $\Ccal_i$ has no constant-bound $Z \oplus G \oplus Y$-trace. In particular, for every $i < k$, $f_i$ is $Z \oplus G \oplus Y$-hyperimmune.
\end{proof}

\section{Immunity and closed sets}

This last section is devoted to the study of two notions of preservation whose hierarchies collapse, namely, preservation of immunities and avoidance of closed sets. These notions are strictly stronger than the notions of preservation we considered so far, and are not known to be distinct.

\begin{lemma}\label{lem:immune-almost-join}
	Fix a set $Z$ and $A_0,A_1,\ldots$ all $Z$-immune. There is a set $G$ which is $Z$-immune and such that for every $n$, $G[n] =^* A_n$.
\end{lemma}
\begin{proof}
 	Write $\omega^{[n]} = \{\la n,x \ra : x \in \omega\}$. We will define $G = \bigoplus_{n \in \omega} G_n$ where $G_n =^* A_n$. We want $G$ to be $Z$-immune, so that no $Z$-computable set is a subset of $G$. Since each $A_n$ is immune, no $Z$-computable infinite set which is a subset of $\omega^{[0]} \cup \cdots \cup \omega^{[n]}$ can be a subset of $G$.
	
	Let $B_0,B_1,\ldots$ be a list of the infinite $Z$-computable sets which intersect infinitely many of the $\omega^{[n]}$. Suppose that we have defined $a_0 < a_1 < \cdots < a_k$ and $G_0,\ldots,G_k$ such that for each $i \leq k$, $B_i \cap \omega^{[a_k]} \nsubseteq \la k , G_k \ra$. Then for some $a_{k+1} > a_k$, $B_{k+1}$ intersects $\omega^{[a_{k+1}]}$, say $\la a_{k+1},n \ra \in B_{k+1}$. Set $G_{i} = A_{i}$ for $a_k < i < a_{k+1}$, and $G_{a_{k+1}} = A_{a_{k+1}} - \{n\}$. So no $B_i$ is a subset of $G$, and hence $G = \bigoplus_{n \in \omega} G_n$ is $Z$-immune.
\end{proof}

\begin{theorem}
	Let $\Psf$ be a problem. Then the following are equivalent:
	\begin{enumerate}
		\item $\Psf$ admits preservation of 1 immunity.
		\item $\Psf$ admits preservation of $\omega$ immunities.
	\end{enumerate}
\end{theorem}
\begin{proof}
	(2)$\Rightarrow$(1) is obvious. For (1)$\Rightarrow$(2): Let $Z$ be a set and $X$ a $Z$-computable instance of $\Psf$. Suppose that $A_1,A_2,\ldots$ are $Z$-immune. Let $G$ be as in Lemma \ref{lem:immune-almost-join}: $G$ is $Z$-immune, and $G[i] =^* A_i$ for all $i$. Then there is a solution $Y$ to $X$ such that $G$ is $Z \oplus Y$-immune. If, for some $i$, $A_i$ was not $Z \oplus Y$-immune, then $Z \oplus Y$ would compute an infinite subset of $A_i$, and hence of $G$ (since $G[i] =^* A_i$). This cannot happen as $G$ is $Z \oplus Y$-immune.
\end{proof}

\begin{lemma}
Preservation of $\omega$ immunities implies avoidance of constant-bound traces for $\omega$ closed sets.
\end{lemma}
\begin{proof}
Suppose a problem $\Psf$ admits preservation of $\omega$ immunities. 
Fix a set $Z$ and a countable collection of closed sets $\Ccal_0, \Ccal_1, \dots \subseteq 2^\omega$ with no $Z$-computable constant-bound trace.
For every $n, k \in \omega$, let $A_{n,k}$ be the set of all finite coded $k$-sets $F$ of binary strings such that every string in $F$ has the same length, and such that $[F] \cap \Ccal_n \neq \emptyset$. Every infinite subset of $A_{n,k}$ computes a $k$-trace of $\Ccal_n$, so $A_{n,k}$ is $Z$-immune. Conversely, every $k$-trace of $\Ccal_n$ computes an infinite subset of $A_{n,k}$. Since $\Psf$ admits preservation of $\omega$ immunities, there is a $\Psf$-solution $Y$ to $X$ such that for every $n, k$, $A_{n,k}$ is $Z \oplus Y$-immune. In particular, for every $n \in \omega$, $\Ccal_n$ has no $Z \oplus Y$-computable constant-bound trace.
\end{proof}

\begin{lemma}\label{lem:closed-set-baire-to-cantor-1}
Fix a set $Z$ and a closed set $\Ccal \subseteq \omega^\omega$ in the Baire space with no $Z$-computable member. There exists a closed set $\Dcal \subseteq 3^\omega$ with no $Z$-computable member, and such that every member of $\Ccal$ computes a member of $\Dcal$.	
\end{lemma}
\begin{proof}
Fix $Z$ and $\Ccal$. Let $T \subseteq 2^{<\omega}$ be a $Z$-computable infinite tree with no $Z$-computable infinite path. Given some $P \in \Ccal$, let $\hat{P} \in 3^\omega$ be defined by $\sigma_02\sigma_12\sigma_22\sigma_32\dots$, where for every $n \in \omega$, $\sigma_n$ is the left-most string in $T$ of length $P(n)$. Let $\Dcal$ be the closure of $\{ \hat{P} : P \in \Ccal \}$. Note that for any $X \in \Dcal \setminus \{ \hat{P} : P \in \Ccal \}$, $X = \sigma\cat Y$ for some $\sigma \in 3^{<\omega}$ and $Y \in [T]$, and that neither $\{ \hat{P} : P \in \Ccal \}$ nor $[T]$ has $Z$-computable members. Therefore $\Dcal$ has no $Z$-computable member. Moreover any $P \in \Ccal$ computes $\hat{P} \in \Dcal$.
\end{proof}

\begin{corollary}
Avoidance of 1 closed set in the Cantor space implies avoidance of $\omega$ closed sets in the Baire space.
\end{corollary}
\begin{proof}
Suppose a problem $\Psf$ admits avoidance of 1 closed set in the Cantor space. 
Fix a set $Z$, countably many closed sets in the Baire space $\Ccal_0, \Ccal_1, \dots \subseteq \omega^\omega$ with no $Z$-computable member, and a $Z$-computable $\Psf$-instance $X$. Let $\Ecal = \{n^\frown P : P \in \Ccal_n \}$. In particular, $\Ecal$ is a closed set with no $Z$-computable member. By~\Cref{lem:closed-set-baire-to-cantor-1}, there is a closed set $\Dcal \subseteq 3^\omega$ with no $Z$-computable member, and such that every member of $\Ecal$ computes a member of $\Dcal$. Let $\tilde{\Dcal} \subseteq 2^\omega$ be the closed set obtained from $\Dcal$ by fixing a binary coding of the ternary strings. In particular, any member of $\Ccal_n$ computes a member of $\tilde{\Dcal}$, and $\tilde{\Dcal}$ has no $Z$-computable members. Since $\Psf$ admits avoidance of 1 closed set in the Cantor space, there is a $\Psf$-solution $Y$ to $X$ such that $\tilde{\Dcal}$ has no $Z \oplus Y$-computable member. In particular, for every $n \in \omega$, $\Ccal_n$ has no $Z \oplus Y$-computable member.
\end{proof}

We now prove that preservation of 1 immunity is strictly above the hierarchy of avoidance of constant-bound traces. Let $\emo$ (Erd\H{o}s-Moser) be the problem whose instances are colorings $f : [\omega]^2 \to 2$. An $\emo$-solution to $f$ is an infinite set $H \subseteq \omega$ such that for every $x < y < z \in H$, and every $i < 2$, if $f(x, y) = i$ and $f(y, z) = i$, then $f(x, z) = i$.

\begin{theorem}
There is a problem that admits avoidance of constant-bound traces for $\omega$ closed sets but not preservation of 1 immunity.
\end{theorem}
\begin{proof}
Patey~\cite{PateyCombinatorial} proved that $\emo$ admits avoidance of constant-bound traces for $\omega$ closed sets. On the other hand, Rice~\cite{RiceThin} constructed a computable instance of $\emo$ such that every solution computes a diagonally non-computable function, while Patey~\cite{Patey2016Partial} constructed a $\Delta^0_2$ immune set $A$ such that every diagonally non-computable function computes an infinite subset of $A$. This shows that $\emo$ does not admit preservation of 1 immunity.
\end{proof}

The following question is left open:

\begin{question}
Does preservation of 1 immunity implies 	avoidance of 1 closed set?
\end{question}

The combinatorics used to prove that a problem admits preservation of 1 immunity and avoidance of 1 closed set are very similar, which could be taken as an argument in favor of a positive answer.

\bibliography{References}
\bibliographystyle{alpha}

\end{document}